\documentclass[11pt]{amsart}
\usepackage[dvipsnames]{xcolor}
\usepackage{amssymb,amsmath,amsthm,enumerate,mathtools,mathptmx}
\usepackage{stmaryrd}
\usepackage[new]{old-arrows}
\usepackage{tikz-cd} 
\usepackage{tikz}
\usepackage{multirow}
\usepackage[utf8]{inputenc}
\usepackage{comment}
\usepackage{bm}			

\usepackage{hyperref}
\hypersetup{%
	colorlinks = true,
	linkcolor = BrickRed,
	citecolor = Green,
	urlcolor	 = blue,
	filecolor = red,
}

\usepackage{cleveref}
\usepackage[inline]{enumitem}
\usepackage[margin=0.9in]{geometry}
\usepackage{parskip}

\DeclareMathAlphabet{\mathcal}{OMS}{cmsy}{m}{n} 
\newtheorem{thm}{Theorem}[section]
\newtheorem*{thm*}{Theorem}
\newtheorem{lem}[thm]{Lemma}
\newtheorem{prop}[thm]{Proposition}
\newtheorem{cor}[thm]{Corollary}

\theoremstyle{definition}
\newtheorem{rem}[thm]{Remark}

\numberwithin{equation}{section}

\crefname{thm}{theorem}{theorems}
\crefname{rem}{remark}{remarks}
\crefname{prop}{proposition}{propositions}
\crefname{lem}{lemma}{lemmas}
\crefname{por}{porism}{porisms}
\crefname{identity}{identity}{identities}
\crefname{equation}{}{}

\DeclareMathOperator{\GL}{GL}
\DeclareMathOperator{\chr}{char}
\DeclareMathOperator{\Hilb}{Hilb}
\DeclareMathOperator{\rank}{rank}
\DeclareMathOperator{\Hom}{Hom}
\DeclareMathOperator{\End}{End}
\DeclareMathOperator{\Stab}{Stab}
\DeclareMathOperator{\sign}{sgn}
\DeclareMathOperator{\Res}{Res}
\DeclareMathOperator{\Spec}{Spec}
\DeclareMathOperator{\Frac}{Frac}

\newcommand{\gfrac}[2]{\genfrac{[}{]}{0pt}{}{#1}{#2}}
\newcommand{\cC}{\check{\mathrm{C}}}
\newcommand{\Cech}{\v{C}ech}
\newcommand{\sym}{\mathcal{S}}
\newcommand{\alt}{\mathcal{A}}
\newcommand{\md}[1]{{\left\lvert #1 \right\lvert}}
\newcommand{\deff}[1]{{\color{Cerulean}#1}}
\newcommand{\andd}{\quad\text{and}\quad}
\newcommand{\into}{\longhookrightarrow}
\newcommand{\smatrix}[1]{\left(\begin{smallmatrix} #1 \end{smallmatrix}\right)}
\newcommand{\DD}[2]{\operatorname{D}_{#1 \,\mid\, #2}}
\newcommand{\PP}[3]{\operatorname{P}_{#1 \,\mid\, #2}^{#3}}
\newcommand{\etale}{\'etale}

\let\mapsto\longmapsto
\let\to\longrightarrow

\begin{document}

\title{Homological properties of invariant rings of permutation groups}

\author{Aryaman Maithani}
\address{Department of Mathematics, University of Utah, 155 South 1400 East, Salt Lake City, UT~84112, USA}
\email{maithani@math.utah.edu}

\thanks{The author was supported by NSF grants DMS~2101671 and DMS~2349623 and a Simons Dissertation Fellowship.} 



\begin{abstract} 
	Consider the action of a subgroup $G$ of the permutation group on the polynomial ring $S \coloneqq k[x_{1}, \ldots, x_{n}]$ via permutations. 
	We show that if $k$ does not have characteristic two, then the following are independent of $k$: 
	the $a$-invariant of $S^{G}$, 
	the property of $S^{G}$ being quasi-Gorenstein, 
	and the Hilbert functions of $H_{\mathfrak{m}}^{n}(S)^{G}$ as well as $H_{\mathfrak{n}}^{n}(S^{G})$; 
	moreover, these Hilbert functions coincide. 
	In particular, 
	being independent of characteristic, they may be computed using characteristic zero techniques,
	such as Molien's formula. 
	In characteristic two, we show that the ring of invariants is always quasi-Gorenstein,
	compute the $a$-invariant explicitly, 
	and show that the Hilbert functions of $H_{\mathfrak{m}}^{n}(S)^{G}$ and $H_{\mathfrak{n}}^{n}(S^{G})$ agree up to a shift, 
	given by 
	the number of transpositions. 
	We determine when the inclusion $S^{G} \longhookrightarrow S$ splits, thereby proving the Shank--Wehlau conjecture for permutation subgroups. 
	Lastly, we determine the ring of $k$-linear differential operators on $S^{G}$, and show that each differential operator lifts to one over $\mathbb{Z}$. 
\end{abstract}

\maketitle

\section{Introduction} \label{sec:introduction}
	
	Let $k$ be a field, and $G$ a subgroup of the permutation group $\sym_{n}$. 
	The group $G$ has a natural action on the polynomial ring $S \coloneqq k[x_{1}, \ldots, x_{n}]$ via $k$-algebra automorphisms determined by
	\begin{equation*} 
		\sigma \cdot x_{i} \coloneqq x_{\sigma(i)}.
	\end{equation*}
	The \deff{invariant subring} is $S^{G} \coloneqq \{f \in S : \sigma(f) = f \text{ for all $\sigma \in G$ }\}$. 
	The invariant rings of permutation groups have a rich history in producing counterexamples to several questions, 
	both within and outside the realm of invariant theory. 

	Outside the realm of invariant theory, 
	Samuel \cite{Samuel:UFD} raised the following question: 
	are all noetherian UFDs Cohen--Macaulay? 
	The first counterexample to this question was given by Bertin \cite{Bertin}; 
	namely the ring of invariants $S^{G}$ of the permutation action of $\langle (1234) \rangle$ on $\mathbb{F}_{2}[w,x,y,z]$. 
	Later, Fossum and Griffith \cite{FossumGriffith:UFDnotCM} showed that the completion of $S^{G}$ at its homogeneous maximal ideal produces an example of a \emph{complete} local UFD that is not Cohen--Macaulay. 

	Within the realm of invariant theory of finite groups, one considers more generally a finite subgroup $G$ of $\GL_{n}(k)$ with its natural action on $S$ via linear change of coordinates. 
	When $G$ is \emph{nonmodular}, i.e., 
	$\chr(k)$ does not divide the order of $G$, 
	one has a wealth of results such as: 
	\begin{enumerate}[nosep]
		\item the inclusion $S^{G} \into S$ splits $S^{G}$-linearly,
		and thus, the invariant ring $S^{G}$ is Cohen--Macaulay; 
		\item (Noether's bound) the invariant ring is generated by invariants of degree at most the order of $G$; 
		and 
		\item there is an $S^{G}$-linear isomorphism 
		$H_{\mathfrak{m}}^{n}(S)^{G} \cong H_{\mathfrak{n}}^{n}(S^{G})$, 
		where $\mathfrak{m}$ and $\mathfrak{n}$ denote the respective homogeneous maximal ideals of $S$ and $S^{G}$. 
	\end{enumerate}
	Each of these is known to be false in the modular case; 
	indeed, one can find counterexamples using permutation groups: 
	for (1) and (2), consider the action of the cyclic permutation group
	$\langle (12)(34)(56) \rangle$ on $S \coloneqq \mathbb{F}_{2}[x_{1}, \ldots, x_{6}]$, 
	and for (3), the action of the alternating group $\alt_{3}$ on $\mathbb{F}_{3}[x_{1}, x_{2}, x_{3}]$ \cite[Example 5.2]{GoelJeffriesSingh}.	

	In \cite{Hashimoto:Frational}, Hashimoto constructs examples of permutation groups $G$ acting on a polynomial ring $S$ such that $S^{G}$ is $F$-rational but not $F$-regular. 
	As a consequence of our results, this cannot happen if $G$ is a subgroup of the alternating group;
	indeed, $S^{G}$ is then quasi-Gorenstein [\Cref{cor:quasi-Gorenstein-characterisation}],
	and an $F$-rational ring is Cohen--Macaulay, and thus, 
	$S^{G}$ would then be Gorenstein, 
	for which the notions of $F$-rationality and $F$-regularity coincide \cite[Theorem 4.2 (g)]{HochsterHuneke:FregTestSmooth}.

	On the other hand, we now draw attention to the fact that 
	permutation groups enjoy some `uniform' properties 
	stemming from the fact that the definition of the action is 
	independent of the ground field.
	The key point here is that $S = k[x_{1}, \ldots, x_{n}]$ has a $k$-basis given by monomials, and that any subgroup $G$ of $\sym_{n}$ permutes the elements of this basis. 
	This simple observation has the immediate consequence that $S^{G}$ has a $k$-basis given by orbit sums of monomials; 
	this description is independent of the ground field $k$, 
	and in particular, we deduce that the Hilbert series of the invariant ring is independent of $k$. 
	This description is also used to prove that $S^{G}$ is $F$-pure when $G \le \sym_{n}$ \cite[page 77]{HochsterHuneke:BCMA}. 
	Hashimoto and Singh \cite{HashimotoSingh} show that 
	$S^{G}$ is $F$-pure and has finite $F$-representation type whenever $G$ acts via permutations, or more generally, via a monomial representation. 
	Many results of this paper are in this spirit: 
	we show that various homological properties and invariants are independent of the ground field --- at least when the characteristic is not two;
	this discrepancy can be attributed to the fact that the character $\sign \colon \sym_{n} \to k^{\times}$ is trivial if $\chr(k) = 2$. 
	Our main results are summarised as follows:

	\begin{thm*}
		Let $k$ be a field, $A$ any ring, and
		$G$ a subgroup of $\sym_{n}$ acting on the polynomial rings
		$S \coloneqq k[x_{1}, \ldots, x_{n}]$ 
		and $S_{A} \coloneqq A[x_{1}, \ldots, x_{n}]$ via permutations. 
		Let $c$ be the number of transpositions in $G$, 
		and $N$ the subgroup generated by the transpositions.
		Let $\mathfrak{m}$ and $\mathfrak{n}$ denote the respective homogeneous maximal ideals of $S$ and $S^{G}$.

		If $\chr(k) \neq 2$, then
		\begin{enumerate}[label=(\arabic*), nosep]
			\item there is an isomorphism of graded $k$-vector spaces 
			$H_{\mathfrak{m}}^{n}(S)^{G} \simeq H_{\mathfrak{n}}^{n}(S^{G})$, and 
			the Hilbert functions of the top local cohomology modules of $S^{G}$ and $S_{\mathbb{Q}}^{G}$ coincide;
			\item the canonical module $\omega_{S^{G}}$ is isomorphic to $S^{G}_{\sign}(-n)$; 
			\item the $a$-invariants of $S^{G}$ and $S_{\mathbb{Q}}^{G}$ coincide, and equal $\deg \Hilb(S^{G})$;
			\item the ring $S^{G}$ is quasi-Gorenstein if and only if $S_{\mathbb{Q}}^{G}$ is quasi-Gorenstein
			if and only if $a(S^{G}) = -(c+n)$; 
			whether the latter holds may be determined using Molien's formula.
		\end{enumerate}

		If $\chr(k) = 2$, then
		\begin{enumerate}[label=(\arabic*), nosep]
			\item there is an isomorphism of graded $k$-vector spaces 
			$H_{\mathfrak{m}}^{n}(S)^{G} \simeq H_{\mathfrak{n}}^{n}(S^{G})(-c)$;
			\item the canonical module $\omega_{S^{G}}$ is isomorphic to $S^{G}(-c-n)$; 
			\item the $a$-invariant of $S^{G}$ equals $-(c+n)$;
			\item the ring $S^{G}$ is quasi-Gorenstein.
		\end{enumerate}

		In any characteristic, we have
		\begin{enumerate}[label=(\arabic*), resume]
			\item the inclusion $S^{G} \into S$ splits $S^{G}$-linearly 
			if and only if 
			$\chr(k)$ does not divide $\md{G/N}$;
			\item the natural map 
			$
				\DD{S_{\mathbb{Z}}^{G}}{\mathbb{Z}} \otimes_{\mathbb{Z}} k
				\to
				\DD{S^{G}}{k}
			$
			is an isomorphism, 
			i.e., each $k$-linear differential operator on $S^{G}$ arises from a $\mathbb{Z}$-linear one on $S_{\mathbb{Z}}^{G}$. 
		\end{enumerate}
	\end{thm*}
	The above statements are proven as
	\begin{enumerate*}[label=(\arabic*)]
		\item \Cref{cor:hilbs-invariants-cohomologies};
		\item \Cref{cor:canonical-permutation};
		\item \Cref{cor:a-invariant}, which also describes the $a$-invariant in the odd characteristic;
		\item \Cref{thm:quasi-Gorenstein-independent} and \Cref{cor:quasi-Gorenstein-characterisation};
		\item \Cref{thm:characterise-splitting};
		\item \Cref{cor:diff-opps-lift}. 
	\end{enumerate*}

	Results (1)--(3) above are motivated by the work of Goel, Jeffries, and Singh \cite{GoelJeffriesSingh} 
	where they characterise when the equality $a(S^{G}) = a(S)$ holds
	and show the following [ibid. Corollary 5.1]: 
	if $G$ is a cyclic subgroup of $\GL_{n}(k)$ without transvections acting on $S = k[x_{1}, \ldots, x_{n}]$, 
	then 
	$H_{\mathfrak{m}}^{n}(S)^{G} \simeq H_{\mathfrak{n}}^{n}(S^{G})$ 
	as graded $k$-vector spaces. 
	While they show that the ``cyclic'' hypothesis cannot be dropped in general,
	our result (1) above shows that it may be dropped for permutation groups. 
	We note that permutation groups contain transvections only in characteristic two, in which case these are precisely the transpositions, accounting for the shift. 

	As a consequence of (5) above, 
	we prove the Shank--Wehlau conjecture \cite[Conjecture 1.1]{ShankWehlau:ModularTransfer} for permutation groups; 
	the general conjecture may be stated as follows: 
	Suppose $G \le \GL_{n}(k)$ is a finite $p$-group acting linearly on the ring $S \coloneqq k[x_{1}, \ldots, x_{n}]$, where $\chr(k) = p$. 
	Then, $S^{G} \into S$ splits if and only if $S^{G}$ is a polynomial ring. 
	This formulation is due to Broer \cite[Corollary 4]{Broer:DirectSummandProperty}, 
	who then proved the conjecture for abelian $p$-groups \cite{Broer:AbelianTransvection}. 
	Since then, this has been extended for some classes of $p$-groups by Elmer and Sezer \cite{ElmerSezer}, and by Kummini and Mondal \cite{KumminiMondal:Nakajima,KumminiMondal:PolynomialInvariantRings,KumminiMondal:ModularRamification}.

	Result (6) above addresses the following question of Smith and Van den Bergh, suitably reformulated: \newline
	\cite[Question 5.1.2]{SmithVanDenBergh}
	Let $A$ be a domain that is a finitely generated $\mathbb{Z}$-algebra. 
	Suppose $R$ is a finitely generated $A$-algebra such that 
	$R \otimes_{A} \Frac(A)$ is the ring of invariants for a linear action of a reductive group on a polynomial ring of characteristic zero. 
	Does there exist a nonempty open subset $U$ of $\Spec A$ 
	such that for each maximal ideal $\mathfrak{m} \in U$, 
	each differential operator in 
	$\DD{R/\mathfrak{m} R}{A/\mathfrak{m} A}$ lifts to
	$\DD{R}{A}$?

	The answer is negative in general: 
	for several classical invariant rings, Jeffries and Singh \cite{JeffriesSingh:Lift} construct differential operators in positive characteristic that do not lift to characteristic zero. 
	Our result (6) above says that the answer is positive for permutation actions, i.e., differential operators \emph{do} lift. 

\section*{Acknowledgments}
	The author thanks Jack Jeffries and Anurag K. Singh for several interesting discussions. 

\section{Preliminaries}
	
	We collect some definitions and facts that will be used later. 
	Throughout, 
	we use $k$ to denote a field, 
	$n$ a positive integer, 
	$S \coloneqq k[x_{1}, \ldots, x_{n}]$ the polynomial ring with its standard grading, 
	and
	$\GL \coloneqq \GL_{n}(k)$ the general linear group with its natural degree-preserving action on $S$. 
	The homogeneous maximal ideal of $S$ is denoted by $\mathfrak{m}$, 
	and $H_{\mathfrak{m}}^{n}(S)$ is the top local cohomology of $S$ supported at $\mathfrak{m}$. 
	Given a finite subgroup $G$ of $\GL$, 
	the natural action of $\GL$ on $S$ restricts to one of $G$. 

	\subsection{Graded modules}
		
		Given a graded module $M$ and an integer $d$, we denote the $d$-th graded component of $M$ by $[M]_{d}$. 
		By $M(i)$ we shall mean the graded module given by $[M(i)]_{d} \coloneqq [M]_{i + d}$. 
		All graded modules that we consider will be component-wise finite-dimensional, i.e., 
		$\rank_{k} [M]_{d} < \infty$ for all $d \in \mathbb{Z}$. 
		The \deff{Hilbert function} of $M$ is the function 
		$\mathbb{Z} \to \mathbb{N}$ given by $d \mapsto \rank_{k} [M]_{d}$. 
		If $M$ and $N$ are graded modules with the same Hilbert function, then they are isomorphic as graded $k$-vector spaces, and we write this as $M \simeq N$.
		If $[M]_{d} = 0$ for $d < 0$, then the \deff{Hilbert series} of $M$ is defined to be the power series
		\begin{equation*} 
			\Hilb(M, t) \coloneqq \sum_{d \ge 0} \rank_{k}([M]_{d}) t^{d} \in \mathbb{Q}[\![t]\!].
		\end{equation*}
		The Hilbert--Serre theorem asserts that the above power series is a rational function when $M$ is a finitely generated module. 
		Writing $\Hilb(M) = f/g$ for polynomials $f$ and $g$, we define the \deff{degree} $\deg \Hilb(M)$ to be the difference 
		$\deg(f) - \deg(g)$.

	\subsection{The twisted group algebra}
		
		The \deff{twisted group algebra} $S \ast G$ is the graded $k$-algebra defined as follows: 
		letting $k G$ denote the group algebra,
		we have 
		$S \ast G \coloneqq S \otimes_{k} k G$
		as a $k$-vector space, 
		with the multiplication rule
		\begin{equation*} 
			(s \otimes \sigma) \cdot (t \otimes \tau) \coloneqq 
			s \sigma(t) \otimes \sigma \tau.
		\end{equation*}

		The above inherits a natural grading from that on $S$ and $k G$, where the latter lives in degree zero. 
		Concretely, if $s \in [S]_{d}$ and $\sigma \in G$, 
		then the simple tensor $s \otimes \sigma$ is homogeneous of degree $d$. 

		Here is an alternative way to think about $S \ast G$-modules: 
		the data of an $S \ast G$-module $M$ is precisely the data of an $S$\nobreakdash-module $M$ equipped with a $k$-linear action of $G$ satisfying 
		$\sigma(s m) = \sigma(s) \sigma(m)$ for all 
		$(\sigma, s, m) \in G \times S \times M$. 
		Moreover, $M$ is graded as an $S \ast G$-module if it is so as an $S$-module and the $G$-action preserves the grading. 
		In particular, $[M]_{d}$ is then a $k G$\nobreakdash-module for each $d \in \mathbb{Z}$. 
		Similarly, 
		a function $f \colon M \to N$ between $S \ast G$-modules is $S \ast G$-linear 
		precisely if it is $S$-linear and $G$-equivariant; 
		the latter means that we have 
		$f(\sigma(m)) = \sigma(f(m))$ for all $(\sigma, m) \in G \times M$.

		\subsection{Twisted representation and semi-invariants}
		
		Given a character, i.e., a homomorphism $\chi \colon G \to k^{\times}$, 
		and a graded $S \ast G$-module $M$, 
		we define $M \otimes \chi$ to be the graded $S \ast G$-module $M \otimes_{k} k$ with the actions 
		\begin{equation*} 
			s \cdot (m \otimes 1) \coloneqq s m \otimes 1
			\andd 
			\sigma \cdot (m \otimes 1) \coloneqq \sigma(m) \otimes \chi(\sigma)
		\end{equation*}
		for $(\sigma, s, m) \in G \times S \times M$. 
		In particular, $M$ and $M \otimes \chi$ are isomorphic as graded $S$-modules. 

		The module of \deff{$\chi$-semi-invariants} is defined as
		\begin{equation*} 
			S^{G}_{\chi} \coloneqq 
			\{s \in S : \sigma(s) = \chi(\sigma) s \text{ for all } \sigma \in G\},
		\end{equation*}
		which is isomorphic to $(S \otimes \chi^{-1})^{G}$ as a graded $S^{G}$-module. 

		\subsection{Duals}
		
		Given a finite $k G$-module $V$, we define its dual $V^{\vee}$ to be the $k$-vector space $\Hom_{k}(V, k)$ with the action of $G$ given by
		\begin{equation*} 
			(\sigma \cdot f)(v) \coloneqq f(\sigma^{-1} v)
		\end{equation*}
		for $(\sigma, v, f) \in G \times V \times V^{\vee}$. 
		This makes $V^{\vee}$ a $k G $-module.

		Next, given a graded $S \ast G$-module $M$, we define its graded dual $M^{\ast}$ to be the following graded $S \ast G$-module:
		the degree $d$ component is the $k G$-module $([M]_{-d})^{\vee} = \Hom_{k}([M]_{-d}, k)$, 
		and given homogeneous elements $s \in [S]_{e}$ and $f \in [M^{\ast}]_{d}$, we define $(s \cdot f) \in [M^{\ast}]_{d + e}$ via
		\begin{flalign*} 
			\phantom{\text{for $m \in [M]_{-d-e}$. }}
			&& (s \cdot f)(m) &\coloneqq f(s \cdot m), 
			&
			\text{for $m \in [M]_{-d-e}$. }
		\end{flalign*}

		Both the duals above have an obvious definition on maps, yielding contravariant endofunctors on the categories of $k G$\nobreakdash-modules and graded $S \ast G$-modules, respectively. 

		We shall also use $(-)^{\ast}$ as an endofunctor on the category of graded $S$-modules and $S^{G}$-modules, with the analogous definitions. 

	\subsection{Pseudoreflections}
		
		An element $\sigma \in \GL$ is a \deff{pseudoreflection} if $\rank(1 - \sigma) = 1$, 
		in which case, $\ker(1 - \sigma)$ is a hyperplane. 
		A non-diagonalisable pseudoreflection is a \deff{transvection}. 
		The element $\smatrix{0 & 1 \\ 1 & 0}$ is a pseudoreflection in every characteristic, 
		but a transvection precisely in characteristic two. 
		More generally, 
		viewing $\sym_{n}$ as a subgroup of $\GL_{n}(k)$ via permutation matrices, 
		we see that the pseudoreflections are exactly the transpositions, 
		and that these are transvections precisely when $\chr(k) = 2$. 

	\subsection{Local cohomology}
		
		In this subsection, we record how the group $\GL$ acts on the top local cohomology module $H_{\mathfrak{m}}^{n}(S)$. 
		Following Kunz \cite[\S4]{Kunz:ResiduesDuality}, we use the notion of generalised fractions to represent the elements of local cohomology. 
		A generalised fraction is an element of the form
		\begin{equation*} 
			\gfrac{f}{t_{1}, \ldots, t_{n}},
		\end{equation*}
		where $f \in S$, 
		and $t_{1}, \ldots, t_{n}$ is a homogeneous system of parameters for $S$. 
		The above fraction represents (the cohomology class of) the element 
		$\frac{f}{t_{1} \cdots t_{n}}$ in the \Cech\ complex $\cC^{\bullet}(t_{1}, \ldots, t_{n})$. 
		If $t_{1}', \ldots, t_{n}'$ is another such system of parameters contained in $(t_{1}, \ldots, t_{n}) S$, 
		and $A$ is an $n \times n$ matrix over $S$ satisfying
		\begin{equation} \tag{TL} \label{eq:transformation-law}
			\begin{bmatrix}
				t_{1}' \\ 
				\vdots \\
				t_{n}' \\
			\end{bmatrix}
			=
			A
			\begin{bmatrix}
				t_{1} \\ 
				\vdots \\
				t_{n} \\
			\end{bmatrix},
			\quad
			\text{then}
			\quad
			\gfrac{f}{t_{1}, \ldots, t_{n}} = 
			\gfrac{\det(A) f}{t_{1}', \ldots, t_{n}'},
		\end{equation}
		by \cite[Theorem 4.18]{Kunz:ResiduesDuality}. 
		Using this, any two generalised fractions can be modified to have a common denominator. 
		We remark that the transformation rule above, in particular, implies that the order of the elements in the denominator is relevant: 
		swapping two elements in the denominator changes the fraction by a sign. 

		The action of $\GL$ on $H_{\mathfrak{m}}^{n}(S)$ is given in the natural way: 
		given $\sigma \in \GL$ and $t \in S$, we first obtain an $S^{G}$-linear map $\sigma$ on the localisations $ S_{t} \to S_{\sigma(t)}$ given by $f/t^{N} \mapsto \sigma(f)/\sigma(t)^{N}$. 
		Extending this, 
		we obtain an $S^{G}$-linear map of \Cech\ complexes $\sigma \colon \cC(t_{1}, \ldots, t_{n}) \to \cC(\sigma(t_{1}), \ldots, \sigma(t_{n}))$, 
		giving us
		\begin{equation*} 
			\sigma \cdot 
			\gfrac{f}{t_{1}, \ldots, t_{n}}
			=
			\gfrac{\sigma(f)}{\sigma(t_{1}), \ldots, \sigma(t_{n})}
		\end{equation*}
		for $(\sigma, f) \in G \times S$, 
		and $t_{1}, \ldots, t_{n}$ any homogeneous system of parameters. 

		\begin{thm}[Graded duality]
			There is a $\GL$-equivariant isomorphism of graded $S$-modules 
			\begin{equation*} 
				(S \otimes \det)^{\ast} \cong H_{\mathfrak{m}}^{n}(S)(-n),
			\end{equation*}
			where $\det = \det_{[S]_{1}}$.
		\end{thm}
		We clarify the above notation regarding the character $\det$, as it may differ in the literature: 
		often, $S$ arises as the symmetric algebra on the dual space $V^{\ast}$, in which case, we have
		$\det_{[S]_{1}} = \det_{V}^{-1}$. 
		To be explicit, if $\sigma \in \GL$ acts on $S$ via 
		$x_{i} \mapsto \alpha_{i} x_{i}$, 
		then we have $\det_{[S]_{1}}(\sigma) = \alpha_{1} \cdots \alpha_{n}$.

		For the purposes of this paper, we only use that the above isomorphism is $\sym_{n}$-equivariant, for which the result follows directly from the transformation law \Cref{eq:transformation-law}. 
		We expect that the above is likely well-known to experts but failing to find a reference in the literature, we provide a proof below. 

		\begin{proof} 
			Graded duality \cite[Theorem 18.6]{24Hours} tells us that $S^{\ast} \cong H_{\mathfrak{m}}^{n}(S)(-n)$ as graded $S$-modules; 
			hence, there is a graded $S$-isomorphism
			\begin{equation*} 
				\Phi \colon H_{\mathfrak{m}}^{n}(S)(-n) \to (S \otimes \det)^{\ast}.
			\end{equation*}
			Setting $H \coloneqq H_{\mathfrak{m}}^{n}(S)(-n)$ 
			and $D \coloneqq (S \otimes \det)^{\ast}$, 
			we wish to show that $\Phi(\sigma(\eta)) = \sigma(\Phi(\eta))$ for all $\sigma \in \GL$ and all homogeneous $\eta \in H$. 
			If $\eta$ is a socle element, i.e., of degree zero, 
			then this is clear in view of the transformation law \Cref{eq:transformation-law} because the action of $\sigma$ on both ${[H]}_{0}$ and ${[D]}_{0}$ is scaling by $\det(\sigma^{-1})$. 
			Next, fix $d > 0$, $\eta \in [H]_{-d}$, and $\sigma \in \GL$. 
			Then, given any $s \in [S]_{d}$, 
			the element $s \eta$ is in the socle, and thus, 
			\begin{equation*} 
				\Phi(\sigma(s \eta)) = \sigma(\Phi(s \eta)).
			\end{equation*}
			Using the $S$-linearity of $\Phi$ and the fact that $H$ and $D$ are $S \ast \GL$-modules, we rewrite both sides of the above equation to obtain
			\begin{align*} 
				\Phi(\sigma(s \eta)) &= \sigma(\Phi(s \eta)) \\
				\Rightarrow 
				\Phi(\sigma(s) \sigma(\eta)) &= \sigma(s \Phi(\eta)) \\
				\Rightarrow
				\sigma(s) \Phi(\sigma(\eta)) &= \sigma(s) \sigma(\Phi(\eta)) \\
				\Rightarrow 
				\sigma(s) [\Phi(\sigma(\eta)) &- \sigma(\Phi(\eta))] = 0
			\end{align*}
			for all $s \in [S]_{d}$. 
			Thus, the element $\delta \coloneqq [\Phi(\sigma(\eta)) - \sigma(\Phi(\eta))] \in [D]_{-d}$ is annihilated by $[S]_{d}$. 
			Because the pairing 
			$[S]_{d} \times [D]_{-d} \to [D]_{0}$ is nondegenerate,
			we get $\delta = 0$, giving us the desired $\GL$-equivariance. 
		\end{proof}

		\begin{cor} \label{cor:invariants-of-H-general}
			As graded $S^{G}$-modules, one has
			\begin{equation*} 
				\pushQED{\qed} 
				{\left((S \otimes \det)^{\ast}\right)}^{G}
				\cong
				H_{\mathfrak{m}}^{n}(S)^{G}(-n).
				\qedhere
				\popQED
			\end{equation*}
		\end{cor}

\section{Twisted permutation representations} \label{sec:twisted-permutation-representation}
	
	Let $\chi \colon G \to k^{\times}$ be a homomorphism, and
	$V$ a finite $k G$-module. 
	We say that $V$ is a \deff{$\chi$-permutation representation} if there exists a $k$-basis $B = \{e_{1}, \ldots, e_{n}\}$ of $V$ and an action of $G$ on $[n] \coloneqq \{1, \ldots, n\}$ such that the action of $G$ on $V$ is given by
	\begin{equation*} 
		\sigma(e_{i}) = \chi(\sigma) e_{\sigma(i)}
	\end{equation*}
	for $\sigma \in G$; 
	such a basis $B$ is a \deff{$\chi$-basis}. 
	For the remainder of this section, we fix $V$ and $B$ as above. 

	An element $i \in [n]$ is \deff{$\chi$-good} if the stabiliser of $i$ is contained in the kernel of $\chi$, 
	i.e., 
	$\Stab_{G}(i) \subset \ker(\chi)$. 
	The orbit $G \cdot i \subset [n]$ is \deff{$\chi$-good} if $i$ is $\chi$-good; this is independent of the orbit representative. 
	Fix orbit representatives $a_{1}, \ldots, a_{r} \in [n]$ such that we have a disjoint orbit decomposition
	\begin{equation*} 
		[n] = G \cdot a_{1} \sqcup \cdots \sqcup G \cdot a_{r}.
	\end{equation*}
	Let $s \le r$ be the number of $\chi$-good orbits, and assume without loss of generality that $a_{1}, \ldots, a_{s}$ are $\chi$-good. 
	For each such $i \in [s]$, we define the \emph{twisted orbit sum}
	\begin{equation*} 
		X(a_{i}) \coloneqq \sum_{\sigma \in G/\Stab_{G}(a_{i})} \chi(\sigma) e_{\sigma(a_{i})} \in V,
	\end{equation*}
	where the above is well-defined (independent of the choice of coset representatives) because $a_{i}$ is $\chi$-good.
	A routine verification shows that $X(a_{i}) \in V^{G}$ for each $i \in [s]$. 
	As the distinct $X(a_{i})$ involve disjoint sets of basis vectors, we see that the set $C \coloneqq \{X(a_{i}) : i \in [s]\}$ is $k$-linearly independent. 
	If $f = \sum_{j = 1}^{n} \alpha_{j} e_{j} \in V^{G}$, 
	then for every $\sigma \in G$ and $j \in [n]$, 
	we must have $\alpha_{\sigma(j)} = \alpha_{j} \chi(\sigma)$. 
	In particular, if $j$ is not $\chi$-good, then choosing $\sigma \in \Stab_{G}(j) \setminus \ker(\chi)$ shows that $\alpha_{j} = 0$. 
	Thus, the condition gives us
	\begin{equation*} 
		f = \sum_{i = 1}^{s} \alpha_{a_{i}} X(a_{i}).
	\end{equation*}
	In other words, $C$ is a $k$-basis for $V^{G}$, giving us:

	\begin{prop} \label{prop:dimension-good-orbits}
		The dimension of $V^{G}$ is the number of $\chi$-good orbits.
		\qed
	\end{prop}

	Next, we analyse the dual module $V^{\vee}$. 
	Let $\{e_{1}^{\ast}, \ldots, e_{n}^{\ast}\}$ be the dual basis for $V^{\vee}$, 
	determined by
	\begin{equation*} 
		e_{i}^{\ast}(e_{j}) = \delta_{i, j},
	\end{equation*}
	where $\delta$ denotes the Kronecker delta. 
	Then, the function $\sigma \cdot e_{i}^{\ast} \in V^{\vee}$ is given as
	\begin{align*} 
		(\sigma \cdot e_{i}^{\ast})(e_{j}) 
		= e_{i}^{\ast}(\sigma^{-1}(e_{j}))
		&= e_{i}^{\ast}(\chi(\sigma^{-1})(e_{\sigma^{-1}(j)})) \\
		&= \chi(\sigma^{-1}) \delta_{i, \sigma^{-1}(j)}
		= \chi(\sigma^{-1}) \delta_{\sigma(i), j}.
	\end{align*}

	Thus, $\sigma \cdot e_{i}^{\ast} = \chi(\sigma)^{-1} e_{\sigma(i)}^{\ast}$; 
	letting $\chi^{-1} \colon G \to k^{\times}$ denote the multiplicative-inverse gives us:

	\begin{prop} \label{prop:dual-permutation-rep}
		The dual $V^{\vee}$ is a $\chi^{-1}$-permutation representation with 
		$\{e_{1}^{\ast}, \ldots, e_{n}^{\ast}\}$
		being a $\chi^{-1}$-basis. 
		There is a $k$-vector space isomorphism $(V^{G})^{\vee} \cong (V^{\vee})^{G}$.
		If $\chi$ takes values in $\{\pm 1\}$, then the map $e_{i} \mapsto e_{i}^{\ast}$ 
		induces a $k G$-module isomorphism. 
	\end{prop}
	\begin{proof} 
		Only the second statement needs a proof, for which we note that $i \in [n]$ is $\chi$-good if and only if it is $\chi^{-1}$-good, and thus, the dimensions of $V^{G}$ and $(V^{\vee})^{G}$ are the same. 
	\end{proof}

	\begin{rem}
		In the modular case, one cannot always commute duals and fixed points, as this example shows: Consider the group 
		$G$ generated by 
		$\smatrix{1 & 1 & 0 \\ 0 & 1 & 0 \\ 0 & 0 & 1}$
		and
		$\smatrix{1 & 0 & 1 \\ 0 & 1 & 0 \\ 0 & 0 & 1}$
		with its canonical action on $V \coloneqq k^{3}$, 
		and corresponding dual action on $V^{\vee}$. 
		We have $\dim V^{G} = 1 \neq 2 = \dim (V^{\vee})^{G}$. 
		Indeed, $V^{G}$ is spanned by $e_{1}$ and $(V^{\vee})^{G}$ by $\{e_{2}^{\ast}, e_{3}^{\ast}\}$.
	\end{rem}

\section{Invariants of local cohomology and local cohomology of invariants}

	For the remainder of this section, we fix $G$ to be a subgroup of $\sym_{n}$ with its permutation action on the polynomial ring $S \coloneqq k[x_{1}, \ldots, x_{n}]$.
	In this case, the character $\det$ coincides with the usual character $\sign \colon \sym_{n} \to k^{\times}$. 
	Moreover, 
	in the terminology of \Cref{sec:twisted-permutation-representation}, 
	each graded component $[(S \otimes \det)]_{d} = [(S \otimes \sign)]_{d}$ is a $\sign$-permutation representation of $G$, 
	with the monomials of degree $d$ being a $\sign$-basis. 
	When $\chr(k) = 2$, $\sign$ is trivial, and thus, every monomial is $\sign$-good; 
	in other characteristics, the $\sign$-good monomials are precisely those whose stabiliser is contained in the alternating group $\alt_{n}$. 
	This observation has the following immediate consequences:

	\begin{thm} \label{thm:sign-molien}
		Let $G \le \sym_{n}$ act on $S$ by permutations. 
		If $\chr(k) \neq 2$, then the Hilbert series of $S^{G}_{\sign}$ is given as
		\begin{equation*} 
			\Hilb(S^{G}_{\sign}, t) = 
			\frac{1}{\md{G}} 
			\sum_{\sigma \in G}
			\frac{\sign(\sigma)}{\det(1 - \sigma t)}.
		\end{equation*}
		If $\chr(k) = 2$, then $S^{G}_{\sign} = S^{G}$, and the Hilbert series is given as
		\begin{equation*} 
			\Hilb(S^{G}_{\sign}, t) = 
			\frac{1}{\md{G}} 
			\sum_{\sigma \in G}
			\frac{1}{\det(1 - \sigma t)}.
		\end{equation*}
		The right hand sides of both the equations above are to be interpreted as elements of $\mathbb{Q}(t)$. 
	\end{thm}
	\begin{proof} 
		It is known that for permutation actions, the Hilbert series of $S^{G}$ can be computed over characteristic zero using Molien's formula, see \cite[Proposition 4.3.4]{Smith:PolynomialInvariantsBook}. 
		This proves the statement for characteristic two. 

		Now let us assume that $\chr(k) \neq 2$. 
		As noted, the $\sign$-good monomials are then independent of characteristic, and thus, 
		we may
		compute the Hilbert series using Molien's formula for semi-invariants in characteristic zero \cite[Theorem~2.5.3]{Benson:PolynomialInvariantsBook}. 
	\end{proof}

	\begin{thm} \label{thm:invariants-of-cohomology-dual-sign}
		If $G \le \sym_{n}$ acts on $S$ via permutations, then we have an isomorphism of graded $k$-vector spaces
		\begin{equation*} 
			H_{\mathfrak{m}}^{n}(S)^{G}(-n) \simeq 
			{(S^{G}_{\sign})}^{\ast}.
		\end{equation*}
	\end{thm}
	Coupled with \Cref{thm:sign-molien}, this yields a formula for the Hilbert function of $H_{\mathfrak{m}}^{n}(S)^{G}$.
	\begin{proof} 
		In view of \Cref{cor:invariants-of-H-general}, one only needs to prove that 
		\begin{equation*} 
			{\left((S \otimes \sign)^{\ast}\right)}^{G} \simeq
			{\left((S \otimes \sign)^{G}\right)}^{\ast},
		\end{equation*}
		and this follows from \Cref{prop:dual-permutation-rep}.
	\end{proof}

	Our next goal is to compare $H_{\mathfrak{m}}^{n}(S)^{G}$ with 
	$H_{\mathfrak{n}}^{n}(S^{G})$, 
	where $\mathfrak{n}$ is the homogeneous maximal ideal of $S^{G}$. 
	To this end, we define the \deff{canonical module} of $S^{G}$ as $\omega_{S^{G}} \coloneqq H_{\mathfrak{n}}^{n}(S^{G})^{\ast}$. 

	We follow \cite{Broer:DirectSummandProperty} to compute the canonical module. 
	By [ibid. Corollary 7], we have
	\begin{equation} \label{eq:canonical-sign-chi}
		\omega_{S^{G}} \cong S^{G}_{\sign/\chi} \theta (-n),
	\end{equation}
	where $\theta \in S$ will be described below, 
	and $\chi \colon G \to k^{\times}$ is the homomorphism defined by
	\begin{equation} \label{eq:definition-chi}
		\chi(\sigma) \coloneqq \frac{\sigma(\theta)}{\theta}.
	\end{equation}

	Let $T \subset G$ be the subset of all transpositions in $G$, and $N \coloneqq \langle T \rangle$ the subgroup generated by the transpositions. 

	We now describe $\theta$ following \cite[\S2.6]{Broer:DirectSummandProperty}: 
	Let $W_{1}, \ldots, W_{s}$ be the different hyperplanes that arise as fixed points of elements of $G$. 
	For each $W_{i}$, 
	let $\alpha_{i} \in [S]_{1}$ be a nonzero linear form that vanishes on $W_{i}$, 
	and let $H_{i} \le G$ be the pointwise stabiliser of $W_{i}$. 
	Let $p$ be the characteristic of $k$ if this is positive, else let $p = 1$. 
	Then, we may write $\md{H_{i}} \coloneqq e_{i} p^{a_{i}}$ with $e_{i}$ coprime to $p$. 
	Because our representation is defined over the prime subfield, we have $\theta = \alpha_{1}^{m_{1}} \cdots \alpha_{n}^{m_{n}}$, 
	where $m_{i} \coloneqq e_{i} + (p - 1) a_{i} - 1$. 

	\begin{prop} \label{prop:description-theta}
		We have
		$\theta = \prod_{(i, j) \in T} (x_{i} - x_{j})$, and thus,
		$\deg \theta = \md{T}$, the number of transpositions in $G$.
	\end{prop}
	\begin{proof} 
		The pseudoreflections in $G$ are precisely the transpositions. 
		Given a transposition $\tau = (i, j)$ in $T$, the pointwise-stabiliser of the hyperplane $V^{\tau}$ is precisely $\langle \tau \rangle$, 
		which has order $2$, 
		and the corresponding linear form defining the hyperplane is $x_{i} - x_{j}$. 
		This gives us the required statement. 
	\end{proof}

	\begin{lem} \label{lem:transposition-flips-sign}
		Suppose $\chr(k) \neq 2$, and $f \in S$ satisfies $\tau(f) = -f$ for all $\tau \in T$. 
		Then, $f \in \theta S$.
	\end{lem}
	\begin{proof} 
		Let $(i, j) \in T$ be a transposition.
		The hypothesis on $f$ implies that $f$ vanishes when $x_{i} = x_{j}$, and in turn, $x_{i} - x_{j}$ divides $f$. 
		As these factors are coprime for distinct transpositions, we get the desired result.
	\end{proof}

	\begin{cor} \label{cor:description-sign-chi}
		If $\chr(k) = 2$, then $S^{G}_{\sign/\chi} = S^{G}$.
		If $\chr(k) \neq 2$,
		then $S^{G}_{\sign/\chi} \theta = S^{G}_{\sign}$. 
	\end{cor}
	\begin{proof} 
		If $\chr(k) = 2$, then $\chi$ must be valued in $\mathbb{F}_{2}^{\times}$, and thus, $\chi$ and $\sign$ are both trivial. 

		Now, suppose that $\chr(k) \neq 2$. 
		By definition, we have that $\theta \in S^{G}_{\chi}$, and thus, $S^{G}_{\sign/\chi} \theta \subset S^{G}_{\sign}$. 
		The reverse inclusion follows similarly, in view of \Cref{lem:transposition-flips-sign}.
	\end{proof}

	\begin{cor} \label{cor:canonical-permutation}
		Let $G \le \sym_{n}$ act on the polynomial ring 
		$S = k[x_{1}, \ldots, x_{n}]$ by permutations. 
		Let $c$ be the number of transpositions in $G$. 
		We have the following isomorphisms as graded $S^{G}$-modules:
		\begin{enumerate}[label=(\arabic*)]
			\item 
			If $\chr(k) = 2$, 
			then $\omega_{S^{G}} \cong S^{G}(-c-n)$; 
			equivalently,
			$H_{\mathfrak{n}}^{n}(S^{G}) \cong (S^{G})^{\ast}(c+n)$.
			\item 
			If $\chr(k) \neq 2$, 
			then $\omega_{S^{G}} \cong S^{G}_{\sign}(-n)$;
			equivalently,
			$H_{\mathfrak{n}}^{n}(S^{G}) \cong (S^{G}_{\sign})^{\ast}(n)$.
		\end{enumerate}
	\end{cor}
	\begin{proof} 
		The descriptions of $\omega_{S^{G}}$ follow from \Cref{eq:canonical-sign-chi} and \Cref{cor:description-sign-chi}, and the rest from $\omega_{S^{G}}^{\ast} \cong H_{\mathfrak{n}}^{n}(S^{G})$.
	\end{proof}

	Putting this together with the isomorphism
	$H_{\mathfrak{m}}^{n}(S)^{G}(-n) \cong (S^{G}_{\sign})^{\ast}$
	from \Cref{thm:invariants-of-cohomology-dual-sign} gives us:

	\begin{cor} \label{cor:hilbs-invariants-cohomologies}
		Let $G \le \sym_{n}$ act on $S$ by permutations. 
		Let $c$ be the number of transpositions in $G$. 
		We have the following isomorphisms as graded $k$-vector spaces: 
		\begin{enumerate}[label=(\arabic*)]
			\item 
			If $\chr(k) = 2$, 
			then $H_{\mathfrak{m}}^{n}(S)^{G} \simeq H_{\mathfrak{n}}^{n}(S^{G})(-c)$.
			\item 
			If $\chr(k) \neq 2$, 
			then $H_{\mathfrak{m}}^{n}(S)^{G} \simeq H_{\mathfrak{n}}^{n}(S^{G})$. \qed
		\end{enumerate}
	\end{cor}

	Recalling that the transpositions are transvections in characteristic two, 
	and that there are no transvections in $\sym_{n}$ in other characteristics,
	letting $t$ denote the number of transvections gives us the
	characteristic-free statement:
	\begin{equation*} 
		H_{\mathfrak{m}}^{n}(S)^{G} \simeq H_{\mathfrak{n}}^{n}(S^{G})(-t),
	\end{equation*}
	as graded $k$-vector spaces.
	Note that while both the objects above are graded $S^{G}$-modules, they need not be isomorphic as $S^{G}$-modules in the modular case, see \cite[Example 5.2]{GoelJeffriesSingh}. 
	We also remark that outside the realm of permutation groups, even an isomorphism as graded vector spaces may not exist, see \cite[Example 5.3]{GoelJeffriesSingh}. 

	We also deduce the $a$-invariant of $S^{G}$ from \Cref{cor:canonical-permutation}. 
	Recall that the \deff{$a$-invariant} is the largest integer $a$ such that $[H_{\mathfrak{n}}^{n}(S^{G})]_{a} \neq 0$. 
	If $S^{G}$ is Cohen--Macaulay, then the $a$-invariant of $S^{G}$ equals the degree of its Hilbert series. 
	These statements can be found in \cite{BrunsHerzog:a-invariants}. 

	\begin{cor} \label{cor:a-invariant}
		Let $G \le \sym_{n}$ act on $S = k[x_{1}, \ldots, x_{n}]$ by permutations. 
		Let $c$ be the number of transpositions in $G$. 
		\begin{enumerate}[label=(\arabic*)]
			\item If $\chr(k) = 2$, then $a(S^{G}) = -(c + n)$.
			\item If $\chr(k) \neq 2$, then $a(S^{G}) = -(d + n)$, where $d$ is the minimal degree of a monomial in $S$ whose stabiliser is contained in the alternating group $\alt_{n}$. 
			We also have $a(S^{G}) = \deg \Hilb(S^{G})$. 
		\end{enumerate}
		In particular, the $a$-invariant is the same across all characteristics not equal to two.
	\end{cor}
	\begin{proof} 
		Statement (1) is clear from \Cref{cor:canonical-permutation}. 
		For (2), we first need to check that the lowest nonzero component of $S^{G}_{\sign}$ is in degree $d$ as defined in the statement of the corollary. 
		This follows from the discussion preceding \Cref{thm:sign-molien}. 
		Thus, $a(S^{G}) = -(d + n)$ is independent of characteristic (not equal to two). 
		In the case of characteristic zero, $S^{G}$ is Cohen--Macaulay, and thus, $a(S^{G}) = \deg \Hilb(S^{G})$. 
	\end{proof}

\section{Quasi-Gorenstein}

	Continuing our earlier hypothesis of $G \le \sym_{n}$ acting on $S \coloneqq k[x_{1}, \ldots, x_{n}]$,
	our next goal is to characterise when $S^{G}$ is \deff{quasi-Gorenstein}, 
	i.e., 
	$\omega_{S^{G}} \cong S^{G}(a)$, 
	for some $a \in \mathbb{Z}$, 
	in which case we must necessarily have $a = a(S^{G})$. 
	Thus, the ring $S^{G}$ is Gorenstein precisely when it is Cohen--Macaulay and quasi-Gorenstein.
	\Cref{cor:canonical-permutation} tells us that $S^{G}$ is always quasi-Gorenstein in characteristic two. 
	In a similar vein as before, we next show that quasi-Gorensteinness is independent of the base field in all other characteristics, and thus, we may reduce to characteristic zero where we have the theorems of Stanley \cite{Stanley:Invariants} and Watanabe \cite{Watanabe:Gorenstein-I-II} characterising the Gorenstein property of invariant rings.	

	\begin{thm} \label{thm:quasi-Gorenstein-independent}
		Let $G \le \sym_{n}$ act by permutations.
		If $\chr(k) \neq 2$, then 
		the following are equivalent:
		\begin{enumerate}[label=(\arabic*)]
			\item The ring $k[x_{1}, \ldots, x_{n}]^{G}$ is quasi-Gorenstein.
			\item The character $\sign/\chi$ is trivial. 
			\item The ring $\mathbb{Q}[x_{1}, \ldots, x_{n}]^{G}$ is quasi-Gorenstein, equivalently, Gorenstein.
		\end{enumerate}
	\end{thm}
	\begin{proof} 
		The equivalence (1) $\Leftrightarrow$ (2) is \cite[Corollary 7 (iv)]{Broer:DirectSummandProperty}. 
		Thus, it suffices to show that the equality $\sign = \chi$ holds over $k$ if and only if it holds over $\mathbb{Q}$. 
		We now analyse $\chi$ using \Cref{prop:description-theta} and \Cref{eq:definition-chi}. 
		For concreteness, we may modify $\theta$ up to a sign and assume that every factor $x_{i} - x_{j}$ of $\theta$ satisfies $i < j$. 
		Then, $\chi(\sigma) = (-1)^{s}$, where $s$ is the cardinality of the set $\{(i, j) \in T : i < j \text{ and } \sigma(i) > \sigma(j)\}$. 
		Thus, both $\sign$ and $\chi$ take values in $\{1, -1\}$, and equality can be checked over any field where $1 \neq -1$. 
	\end{proof}

	\begin{cor} \label{cor:quasi-Gorenstein-characterisation}
		Let $G \le \sym_{n}$ act on $S$ by permutations, and let $c$ be the number of transpositions in $G$. 
		The ring $S^{G}$ is quasi-Gorenstein if and only if $a(S^{G}) = -(c + n)$. 
		More precisely,
		\begin{enumerate}[label=(\arabic*)]
			\item if $\chr(k) = 2$, then $S^{G}$ is quasi-Gorenstein;
			\item if $\chr(k) \neq 2$, 
			then $S^{G}$ is quasi-Gorenstein if and only if $\deg \Hilb(S^{G}) = -(c + n)$. 
			In particular, if $G$ contains no transpositions, then $S^{G}$ is quasi-Gorenstein if and only if $G$ is contained in $\alt_{n}$. 
		\end{enumerate}
	\end{cor}
	\begin{proof} 
		For $\chr(k) = 2$, the statement is a consequence of \Cref{cor:canonical-permutation}. 
		Let us now assume that $\chr(k) \neq 2$. 
		Then, by \Cref{thm:quasi-Gorenstein-independent}, we may assume that $k = \mathbb{Q}$. 
		Equation \Cref{eq:canonical-sign-chi} tells us that $\omega_{S^{G}} \cong S^{G}_{\sign/\chi}(-c-n)$ as $\deg \theta = c$. 
		Thus, we have $a(S^{G}) = -(c + n)$ if and only if $\sign/\chi$ is trivial if and only if $S^{G}$ is quasi-Gorenstein. 
		Finally, if $G$ contains no transpositions, then Watanabe's theorems \cite{Watanabe:Gorenstein-I-II} tell us that $S^{G}$ is (quasi-)Gorenstein precisely when $\det$ is trivial, i.e., 
		precisely when $G \le \alt_{n}$. 
	\end{proof}

\section{Young subgroups} \label{sec:young-subgroups}
	
	In this section, we analyse subgroups of $\sym_{n}$ that are generated by transpositions, also known as \deff{Young subgroups} in the literature. 
	Throughout this section, $N$ denotes a subgroup that is generated by transpositions. 
	Consider the relation $\sim$ defined on $[n]$ by $i \sim j$ if and only if $i = j$ or $(i, j) \in N$. 
	The relation is evidently reflexive and symmetric, it is also transitive in view of the computation $(i, j) (j, k) (i, j) = (i, k)$. 
	We first note that this equivalence relation precisely captures the action of $N$ on $[n]$. 

	\begin{lem} \label{lem:equivalence-classes-N}
		The elements $i, j \in [n]$ are in the same equivalence class of $\sim$ if and only if $i$ in the $N$-orbit of $j$. 
	\end{lem}
	\begin{proof} 
		The nontrivial implication $(\Leftarrow)$ follows from the fact that $N$ is generated by transpositions. 
	\end{proof}

	Given a subset $A \subset [n]$, we define
	\begin{equation*} 
		\sym_{A} \coloneqq \{\sigma \in \sym_{n} : 
		\sigma(i) = i \text{ for all } i \notin A\}.
	\end{equation*}
	Thus, $\sym_{A}$ is a subgroup of $\sym_{n}$ that is abstractly isomorphic to the symmetric group on $\md{A}$ elements. 
	
	Let $[n] = A_{1} \sqcup \cdots \sqcup A_{r}$ be the orbit decomposition under the action on $N$. 
	As the $A_{i}$ form a partition of $[n]$, we see that the group generated by the $\sym_{A_{i}}$ is the internal direct product 
	$\sym_{A_{1}} \times \cdots \times \sym_{A_{r}}$.
	Moreover, by \Cref{lem:equivalence-classes-N}, it follows that this internal direct product is equal to $N$.

	Consider now the action of $N$ on the polynomial ring $S = k[x_{1}, \ldots, x_{n}]$. 
	For $A \subset [n]$, define $E(A) \subset S$ to be the set of elementary symmetric polynomials in the variables $\{x_{i} : i \in A\}$. 
	Then, it is clear that for each orbit $A_{i}$, the set $E(A_{i})$ consists of $N$-invariants. 
	Moreover, the union $\bigcup_{i} E(A_{i})$ has cardinality $n$. 
	By \cite[Theorem 3.10.1]{DerksenKemper}, $S^{N}$ is a polynomial ring on this set. 

	Next, consider an arbitrary subgroup $G \le \sym_{n}$, and let $N \le G$ be the subgroup generated by the transpositions in~$G$. 
	Then, $N$ is a normal subgroup of $G$, and thus, $G$ acts on the set of orbits $\{A_{1}, \ldots, A_{r}\}$. 
	Similarly, considering the permutation action on the polynomial ring 
	$S = k[x_{1}, \ldots, x_{n}]$, 
	we see that $G$ acts on the polynomial ring~$S^{N}$. 
	Moreover, if we write $S^{N} = k[y_{1}, \ldots, y_{n}]$ where $y_{i}$ are the generators obtained from the description of the previous paragraph, then the action of $G$ is a permutation action on these variables: 
	if $\sigma \in G$ and $A_{i}$ is an orbit, 
	then $\sigma \cdot A_{i}$ is another orbit $A_{j}$, 
	and $\sigma$ maps the elementary symmetric polynomials in the $A_{i}$ to the corresponding ones in $A_{j}$. 
	Moreover, as $N$ acts trivially on $S^{N}$, we obtain an action of $G/N$ on $S^{N}$ --- one checks that this permutation action is now without transpositions. 
	We summarise this in the following. 

	\begin{prop} \label{prop:factoring-action-not-small}
		Let $G \le \sym_{n}$ act on $S = k[x_{1}, \ldots, x_{n}]$ via permutations, 
		and let $N \le G$ be the (normal) subgroup generated by the transpositions in $G$. 
		Then, the ring of invariants $S^{N}$ is a polynomial ring generated by the elementary symmetric polynomials in the variables corresponding to the $N$-orbits. 
		The action of $G/N$ on $S^{N}$ is via permutations and without transpositions, and we have
		\begin{equation*} 
			\pushQED{\qed} 
			S^{G} = {(S^{N})}^{G/N}.	\qedhere
			\popQED 
		\end{equation*}
	\end{prop}

\section{The direct summand property}

	In this section, we characterise when $S^{G}$ is a direct summand of $S$, i.e.,
	the inclusion $S^{G} \into S$ splits $S^{G}$-linearly. 
	This always happens in characteristic zero, whereas in positive characteristic, this happens precisely when $S^{G}$ is \emph{$F$-regular}. 
	As a consequence of our characterisation, we show that the Shank--Wehlau conjecture is true for permutation subgroups [\Cref{cor:shank-wehlau-permutations}].

	As before, $G$ is a subgroup of $\sym_{n}$ acting on $S \coloneqq k[x_{1}, \ldots, x_{n}]$ via permutations, and $N$ is the (normal) subgroup generated by the transpositions in $G$.

	\begin{thm} \label{thm:characterise-splitting}
		Let $G \le \sym_{n}$ act on $S = k[x_{1}, \ldots, x_{n}]$ via permutations, 
		and let $N \le G$ be the subgroup generated by the transpositions in $G$. 
		The inclusion $S^{G} \into S$ splits if and only if $\chr(k)$ does not divide $\md{G/N}$.
	\end{thm}
	\begin{proof} 
		By \cite[Theorem 2]{Broer:DirectSummandProperty}, 
		the inclusion $S^{G} \into S$ splits if and only if the inclusion $S^{N} \into S$ splits and $\chr(k)$ does not divide $\md{G/N}$. 
		But $S^{N}$ is regular by \Cref{prop:factoring-action-not-small}, and thus, $S^{N} \into S$ is split. 
	\end{proof}

	\begin{cor} \label{cor:shank-wehlau-permutations}
		If $\chr(k) = p$ and $G \le \sym_{n}$ is a $p$-group such that $S^{G} \into S$ splits, 
		then $S^{G}$ is a polynomial ring.
	\end{cor}
	\begin{proof} 
		By \Cref{thm:characterise-splitting}, we get that $G = N$, and thus, $S^{G} = S^{N}$ is a polynomial ring by \Cref{prop:factoring-action-not-small}. 
	\end{proof}

	\begin{cor}
		If $G \le \sym_{n}$ is such that $k[x_{1}, \ldots, x_{n}]^{G}$ is $F$-regular for all fields $k$ of positive characteristic, 
		then $G$ is generated by transpositions,
		and in turn, $k[x_{1}, \ldots, x_{n}]^{G}$ is a polynomial ring for all fields $k$.
	\end{cor}
	\begin{proof} 
		In view of \Cref{thm:characterise-splitting}, the hypothesis forces $G = N$. 
	\end{proof}

	We compare the above with \cite{Blum-SmithMaques:CM} which says: If $G \le \sym_{n}$ and $k[x_{1}, \ldots, x_{n}]^{G}$ is Cohen--Macaulay for all fields $k$ (of positive characteristic), then $G$ is generated by \emph{bireflections}, i.e., elements $\sigma$ satisfying $\rank(1 - \sigma) \le 2$. 

\section{Differential operators}
	We now discuss differential operators on the ring of permutation invariants, and show that, in a certain sense, these are again independent of the base field. 
	Following \cite{BrennerJeffriesNunezBetancourt}, we recall the definition of differential operators. 
	Let $A$ be a commutative ring, 
	and $R$ a commutative $A$-algebra. 
	Given an $R$-module $M$ and an element $r \in R$, we let
	$r_{M} \colon M \to M$ denote the multiplication-by-$r$ map. 
	For $R$-modules $M$ and $N$, the set of
	\deff{$A$-linear differential operators of order at most $n$} 
	is the subset
	$\DD{R}{A}^{n}(M, N) \subset \Hom_{A}(M, N)$ inductively defined as
	\begin{itemize}
		\item $\DD{R}{A}^{-1}(M, N) \coloneqq 0$;
		\item $\DD{R}{A}^{n}(M, N) \coloneqq 
		\{\delta \in \Hom_{A}(M, N) : 
		\delta \circ r_{M} - r_{N} \circ \delta \in \DD{R}{A}^{n-1}(M, N)
		\text{ for all } r \in R\}$
		for $n \ge 0$. 
	\end{itemize}
	These are $R$-submodules of $\Hom_{A}(M, N)$ with the action of $R$ given by $r \cdot \delta = r_{N} \circ \delta$,
	and the $R$-module of \deff{$A$-linear differential operators} is 
	$\DD{R}{A}(M, N) \coloneqq \bigcup_{n \ge 0} \DD{R}{A}^{n}(M, N)$. 
	Specialising $M = N = R$, one checks that $\DD{R}{A}(R, R)$ is a subring of $\End_{A}(R)$, and we denote this by $\DD{R}{A}$.

	Following \cite[\S5.1]{SmithVanDenBergh}, if $B$ is another $A$-algebra, then there is a natural $B$-linear homomorphism
	\begin{equation*} 
		\DD{R}{A} \otimes_{A} B \to \DD{R \otimes_{A} B}{B}. 
	\end{equation*}
	If $R$ is a flat and finitely generated $\mathbb{Z}$-algebra, 
	and $k$ is any field, then the above gives a $k$-linear injection
	\begin{equation*} 
		\DD{R}{\mathbb{Z}} \otimes_{\mathbb{Z}} k \into \DD{R \otimes_{\mathbb{Z}} k}{k}. 
	\end{equation*} 
	We will show that if $R = \mathbb{Z}[x_{1}, \ldots, x_{n}]^{G}$ for a permutation group $G$, 
	then the above is an isomorphism for every field $k$. 
	In particular, every differential operator on $\mathbb{F}_{p}[x_{1}, \ldots, x_{n}]^{G}$ lifts to one in characteristic zero. 
	We do this by explicitly computing the differential operators and showing that they are obtained as linear combinations of suitable orbit sums. 
	To this end, we first recall a criterion for when differential operators can be lifted to the ambient ring. 
	Recall that a subgroup of $\GL_{n}(k)$ is \deff{small} if it contains no pseudoreflections. 
	The following is a theorem of Kantor \cite{Kantor:1} in the nonmodular case. 
	We state and prove the version we need in the modular case.

	\begin{prop} \label{prop:small-etale-extension} 
		If $G \le \GL_{n}(k)$ is a small subgroup with its natural action on the ring $S \coloneqq k[x_{1}, \ldots, x_{n}]$, 
		then the extension 
		$S^{G} \into S$ is \etale\ in codimension one. 
		Consequently, every differential operator on $S^{G}$ lifts uniquely to one on $S$. 
		More precisely, given $\delta \in \DD{S^{G}}{k}^{n}$, 
		there exists a unique $\partial \in \DD{S}{k}^{n}$ such that 
		$\partial|_{S^{G}} = \delta$. 
	\end{prop}
	\begin{proof} 
		Recall that a finite ring extension $R \into S$ is \deff{\etale} 
		if $S$ is flat over $R$ and
		$\kappa(\mathfrak{p}) \otimes_{R} S$ is a finite product of finite separable algebraic extension fields of $\kappa(\mathfrak{p})$ for every $\mathfrak{p} \in \Spec R$. 
		If $\mathfrak{p} \in \Spec(S^{G})$ is a prime of height one (resp. zero), then the extension $(S^{G})_{\mathfrak{p}} \into S_{\mathfrak{p}}$ is flat as this is a finite extension of DVRs (resp. fields) and the latter is torsion-free. 
		As $G$ is small, the prime $\mathfrak{p}$ is unramified (see, for example, \cite[Theorem 2]{Broer:DirectSummandProperty}), and thus $S_{\mathfrak{p}}/\mathfrak{p} S_{\mathfrak{p}}$ is a separable field extension of $\kappa(\mathfrak{p})$. 
		We thus conclude that $S^{G} \into S$ is \etale\ in codimension one. 

		Suppose now that $R \into S$ is an extension of $k$-algebras that is	\etale\ in codimension one. 
		As in the proof of \cite[Proposition~6.4]{BrennerJeffriesNunezBetancourt}, we note that
		the restriction map $\DD{S}{k}^{n}(S, S) \to \DD{R}{k}(R, S)$ factors as
		\begin{equation*} 
			\DD{S}{k}^{n}(S, S) 
			\cong
			\Hom_{S}(\PP{S}{k}{n}, S)
			\to
			\Hom_{S}(S \otimes_{R} \PP{S}{k}{n}, S)
			\cong
			\Hom_{R}(\PP{R}{k}{n}, S)
			\cong
			\DD{R}{k}^{n}(R, S),
		\end{equation*}
		where $\PP{R}{k}{n}$ is the module of principal parts. 
		The second map is obtained by applying $\Hom_{S}(-, S)$ to the natural map 
		$\PP{R}{k}{n} \leftarrow S \otimes_{R} \PP{S}{k}{n}$, 
		which is an isomorphism in codimension one \cite[Lemma 2.19]{BrennerJeffriesNunezBetancourt}. 
		In turn, the above map is an isomorphism in codimension one. 
		As the source and target are reflexive modules, we are done. 
	\end{proof} 

	If $S$ is a $k$-algebra and $G$ acts on $S$ via $k$-algebra automorphisms, then we obtain an action of $G$ on $\DD{S}{k}$ via conjugation; namely,
	\begin{equation*} 
		(\sigma \cdot \partial)(s) \coloneqq \sigma(\partial(\sigma^{-1}(s)))
	\end{equation*}
	for $(\sigma, \partial, s) \in G \times \DD{S}{k} \times S$. 
	Note that this is simply the restriction of the usual $G$-action on $\Hom_{k}(S, S)$. 
	In turn, we see that that the $G$-invariant operators are precisely the $G$-equivariant operators, i.e., operators satisfying $\sigma(\partial(s)) = \partial(\sigma(s))$. 
	In particular, such an operator $\partial \in (\DD{S}{k})^{G}$ satisfies $\partial(S^{G}) \subset S^{G}$, giving us a $k$-linear map
	\begin{equation*} 
		\Res \colon {(\DD{S}{k})}^{G} \to \DD{S^{G}}{k}. 
	\end{equation*}

	\begin{cor} \label{cor:small-subgroup-diff-invariants}
		If $G \le \GL_{n}(k)$ is a small subgroup with its natural action on the polynomial ring $S \coloneqq k[x_{1}, \ldots, x_{n}]$, 
		then the restriction map $\Res$ is an isomorphism. 
	\end{cor}
	\begin{proof} 
		Injectivity follows directly from \Cref{prop:small-etale-extension}.
		Conversely, given an operator $\delta \in \DD{S^{G}}{k}$, let $\partial \in \DD{S}{k}$ be its extension. 
		We wish to show that $\partial$ is $G$-invariant.
		To this end, given $\sigma \in G$ and $r \in S^{G}$, we note
		\begin{align*} 
			(\sigma \cdot \partial)(r) = \sigma(\partial(\sigma^{-1}(r))) = \sigma(\partial(r)) = \sigma(\delta(r)) = \delta(r).
		\end{align*}
		Thus, $(\sigma \cdot \partial)|_{S^{G}} = \delta$ and \Cref{prop:small-etale-extension} tells us that $\sigma \cdot \partial = \partial$. 
		In other words, $\partial$ is invariant, as desired. 
	\end{proof}

	We recall \cite[Th\'eor\`eme 16.11.2]{EGA4} that if 
	$S \coloneqq A[x_{1}, \ldots, x_{n}]$ is the polynomial ring, 
	then $\DD{S}{A}$ has the $A$-basis:
	\begin{equation*} 
		\{x^{\alpha} \cdot \partial^{\beta}_{x} : \alpha, \beta \in \mathbb{N}^{n}\},
	\end{equation*}
	where 
	$x^{\alpha} \coloneqq x_{1}^{\alpha_{1}} \cdots x_{n}^{\alpha_{n}}$, 
	and $\partial^{\beta}_{x}$ is the $A$-linear operator on $S$ satisfying
	\begin{flalign*} 
		\phantom{\text{for $\gamma \in \mathbb{N}^{n}$.}}
		&& 
		\partial^{\beta}_{x}(x^{\gamma})
		&=
		\binom{\gamma_{1}}{\beta_{1}} 
		\cdots
		\binom{\gamma_{n}}{\beta_{n}} 
		x^{\gamma - \beta}
		&
		\text{for $\gamma \in \mathbb{N}^{n}$.}
	\end{flalign*}
	We refer to the above basis as the 
	\deff{PBW-basis for $\DD{S}{A}$ with respect to $x_{1}, \ldots, x_{n}$}. 

	If $G \le \sym_{n}$ acts via permutations on $S$, then the induced action of $G$ on $\DD{S}{k}$ is via permuting the above basis. 
	Thus, as in \Cref{sec:twisted-permutation-representation}, we see that the invariant subring $(\DD{S}{k})^{G}$ has a $k$-basis given by orbit sums. 
	Combining this with \Cref{cor:small-subgroup-diff-invariants} and \Cref{prop:factoring-action-not-small}, we obtain the following.

	\begin{thm} \label{thm:diff-ops-for-permutation-groups}
		Let $G \le \sym_{n}$ act on $S \coloneqq k[x_{1}, \ldots, x_{n}]$ via permutations, and $R \coloneqq S^{G}$ be the ring of invariants.
		
		If $G$ contains no transpositions, then
		$\DD{R}{k}$ has a $k$-basis consisting of $G$-orbit sums of elements of the PBW-basis with respect to $x_{1}, \ldots, x_{n}$.

		More generally, if $G \le \sym_{n}$ is arbitrary, 
		and $N \unlhd G$ is the subgroup generated by the transpositions in $G$, 
		then $S^{N}$ is the polynomial ring $k[y_{1}, \ldots, y_{n}]$ as described in \Cref{sec:young-subgroups}, 
		and $\DD{R}{k}$ has a $k$-basis consisting of the $G/N$-orbit sums of the elements of the PBW-basis for $\DD{S^{N}}{k}$ with respect to $y_{1}, \ldots, y_{n}$. 
		\qed
	\end{thm}

	\begin{cor} \label{cor:diff-opps-lift}
		Let $G \le \sym_{n}$ act on $S \coloneqq \mathbb{Z}[x_{1}, \ldots, x_{n}]$ via permutations.
		If $R \coloneqq S^{G}$ is the ring of invariants, 
		and~$k$ any field, 
		then the natural map
		\begin{equation*} 
			\Phi \colon 
			\DD{R}{\mathbb{Z}} \otimes_{\mathbb{Z}} k
			\to
			\DD{R \otimes_{\mathbb{Z}} k}{k}
		\end{equation*}
		is an isomorphism.
	\end{cor}
	\begin{proof} 
		We first note that $R \otimes_{\mathbb{Z}} k$ is precisely the ring of invariants $k[x_{1}, \ldots, x_{n}]^{G}$. 
		As before, we may pass to the case where $G$ contains no transpositions. 
		In view of the PBW-basis, we deduce that the natural map 
		$\DD{S}{\mathbb{Z}} \otimes_{\mathbb{Z}} k \to 
		\DD{S \otimes_{\mathbb{Z}} k}{k}$
		is an isomorphism. 
		Taking $G$-invariants induces the desired isomorphism $\Phi$, in view of our description of the invariant differential operators as linear combinations of orbit sums. 
	\end{proof}

\bibliographystyle{amsplain}
\providecommand{\bysame}{\leavevmode\hbox to3em{\hrulefill}\thinspace}
\providecommand{\MR}{\relax\ifhmode\unskip\space\fi MR }
\providecommand{\MRhref}[2]{%
	\href{http://www.ams.org/mathscinet-getitem?mr=#1}{#2}
}
\providecommand{\href}[2]{#2}

\end{document}